\documentclass[a4paper,12pt]{article}
\usepackage{geometry,latexsym,amssymb}
\usepackage[latin5]{inputenc}
\usepackage{enumerate}
\usepackage{enumitem}
\usepackage[usenames]{color}
\usepackage{graphicx}
\usepackage{enumerate}
\usepackage{rotating}
\usepackage{multirow}
\usepackage{cite}
\usepackage{amsthm}
\usepackage{caption}
\usepackage{tikz}
\usetikzlibrary{matrix,arrows}
\usepackage{amsmath,amscd}
\usepackage{calc}
\usepackage{tabularx}
\usepackage{ifthen}

\usepackage{amsfonts}
\usepackage{amsmath}
\usepackage{footnote}
\usepackage{multicol}
\usepackage{booktabs}
\usepackage[flushleft]{threeparttable}
\usepackage[font=small,labelfont=bf]{caption}
\usepackage{lipsum}
\usepackage{eucal}
\usepackage{amssymb,amscd,amsthm, verbatim,amsmath,color,fancyhdr, mathrsfs}
\usepackage{graphicx}
\usepackage{turnstile}
\usepackage[colorinlistoftodos]{todonotes}

\newcommand\blfootnote[1]{%
  \begingroup
  \renewcommand\thefootnote{}\footnote{#1}%
  \addtocounter{footnote}{-1}%
  \endgroup
}
\newtheorem{thm}{Theorem}[section]
\newtheorem{prop}[thm]{Proposition}

\newtheorem{cor}[thm]{Corollary}

\newtheorem{ex}[thm]{Example}
\newtheorem{defn}[thm]{Definition}

\begin{document}

\begin{center}
{\Large \textbf{On Bell numbers of type $D$}} \vspace*{0.5cm}
\end{center}

\begin{center}
Hasan Arslan$^{*,a}$, Nazmiye Alemdar$^{a}$ , Mariam Zaarour$^{b}$, H{\"u}seyin Alt{\i}ndi\c{s}$^{a}$ 
\end{center}
\begin{center}
$^{a}${\small {\textit{Department of  Mathematics, Faculty of Science, Erciyes University, 38039, Kayseri, Turkey}}}\\
$^{b}$ {\small {\textit{Graduate School of Natural and Applied Sciences, Erciyes University, Kayseri, Turkey}}}\\
\end{center}

\blfootnote{Email Adressess: hasanarslan@erciyes.edu.tr (H. Arslan),  nakari@erciyes.edu.tr (N. Alemdar), mariamzaarour94@gmail.com (M. Zaarour), altindis@erciyes.edu.tr (H. Alt{\i}ndi\c{s})\\
*Corresponding Author: Hasan Arslan}

\begin{abstract}
In this paper, we will introduce Bell numbers $D(n)$ of type $D$ as an analogue to the classical Bell numbers related to all the partitions of the set $[n]$. Then based on a signed set partition of type $D$, we will construct the recurrence relations of Bell numbers  $D(n)$. In addition, we deduce the exponential generating function for $D(n)$. Finally, we will provide an explicit formula for $D(n)$.
\end{abstract}

\textbf{Keywords}: Bell number, Stirling number, generating function, signed set partition.\\

\textbf{2020 Mathematics Subject Classification}: 05A05, 05A15, 05A18, 05A19, 11B73.
\\

\section{Introduction}\label{sec1}
The main objective of this paper is to study Bell numbers of type $D$. First of all, we will give some fundamental concepts which we will use   throughout the paper. Let $[n]:=\{1,\cdots, n\}$ and $\langle n\rangle :=\{-n,\cdots, 1, 0, 1, \cdots, n\}$. The partitions of the set $[n]$ into $k$ non-empty subset, are called \textit{blocks}, are enumerated by the \textit{ordinary Stirling numbers of the second kind} and denoted by $S(n,k)$ (see \cite{br4}). The ordinary Stirling numbers of the second kind have the recurrence relation for all $n\geq 1$ 
$$S(n,k)=S(n-1,k-1)+kS(n-1,k)$$
with the initial conditions $S(0,k)=\delta_{0,k}$, where $\delta_{i,j}$ is the Kronecker delta.
The number of all set partitions of the $[n]$ is called \textit{the classical Bell numbers} and is denoted by $A(n)$. Clearly, $A(n)=\sum_{k=0}^nS(n,k)$. Classical Bell numbers $A(n)$ satisfy the following well-known recurrence relation for all $n\geq 0$
$$A(n+1)=\sum_{k=0}^n {n \choose k} A(k).$$
Also, the exponential generating function for $A(n)$ is
$$A(x)=e^{e^x-1}.$$
When taking $A(x)=e^{e^x-1}$ into account, it is not hard to see that a concrete formula for $A(n)$ is given by 
$$A(n)=e^{-1} \sum_{r \geq 0}\frac{r^n}{r!}.$$

One can observe some classical Bell numbers for given a few small $n$ and $k$ values in Table \ref{tab1}. \\

\begin{table}[t]%
	\centering
	\caption{\centering{Bell numbers $A(n)$ and the classical Stirling numbers of the second kind $S(n,k)$}\label{tab1}}%
	\scalebox{0.95}{	\begin{tabular*}{\textwidth}{@{\extracolsep{\fill}}c|c||cccccccc@{\extracolsep{\fill}}}
		\toprule
		n & $A(n)$ & $S(n,0)$ & $S(n,1)$ & $S(n,2)$ & $S(n,3)$ & $S(n,4)$ & $S(n,5)$ & $S(n,6)$ & $S(n,7)$ \\
		\midrule
		0 & 1 & 1 & & & & & & & \\
		1 & 1 & 0 & 1 & & & & & & \\
		2 & 2 & 0 & 1 & 1 & & & & & \\
		3 & 5 & 0 & 1 & 3 & 1 & & & & \\
		4 & 15 & 0 & 1 & 7 & 6 & 1 & & & \\
		5 & 52 & 0 & 1 & 15 & 25 & 10 & 1 & & \\
		6 & 203 & 0 & 1 & 31 & 90 & 65 & 15 & 1 & \\
		7 & 877 & 0 & 1 & 63 & 301 & 350 & 140 & 21 & 1 \\
		\bottomrule
	\end{tabular*}}
\end{table}

Sagan and Swanson \cite{br3} introduced signed set partition of $\langle n\rangle$.  Due to \cite{br3}, a \textit{signed set partition} is a partition of the set $<n>$ having the form
$$\tau=P_0~|~P_1/P_2~|~P_3/P_4~|~\cdots~|~P_{2k-1}/P_{2k}$$
satisfying both
\begin{enumerate}
	\item [$\bullet$]$0\in P_0$ and if $i \in P_0$ then $-i \in P_0$, and 
	\item [$\bullet$] $P_{2i-1}=-P_{2i}$ for all $i,~1 \leq i \leq k$
\end{enumerate}
where $-S=\{-s~:~s \in S\}$.
The blocks $P_{2i-1}$ and $P_{2i}$ are called \textit{paired} and the block $P_0$ is called \textit{zero-block}. Denote by $S_B(<n>,k)$ the set of all type $B$ set partitions of $<n>$ with $2k+1$ blocks.

The Stirling numbers of type $B$ of the second kind denoted by $S_B(n,k)$, which was first defined by Reiner in \cite{br2}, is the number of signed set partitions of $\langle n\rangle$ having exactly $k$ paired blocks and have the following recurrence relation for all $n\geq 1$
\begin{equation}\label{rec}
	S_B(n,k)=S_B(n-1,k-1)+(2k+1)S_B(n-1,k)
\end{equation}
with the initial conditions $S_B(0,k)=\delta_{0,k}$, see \cite{br1}. 
Bell numbers of type $B$, which is also known as Dowling numbers (see \cite{br1'}), and second kind Stirling numbers of type $B$ correspond to the sequences oeis.org/A007405 and oeis.org/A039755 in OEIS, respectively. In Table \ref{tab2}, one could see both Bell numbers of type $B$ and Stirling numbers in type $B$ of the second kind for small values of $n$ and $k$.\\

\begin{table}[h]%
	\centering
	\caption{\centering{Bell numbers $B(n)$ and Stirling numbers $S_B(n,k)$ in type $B$ of the second kind} \label{tab2}}%
            \scalebox{0.8}{
	\begin{tabular}{c|c||cccccccc}
		\toprule
		n &$B(n)$& $S_B(n,0)$ & $S_B(n,1)$ &$S_B(n,2)$&$S_B(n,3)$& $S_B(n,4)$& $S_B(n,5)$&$S_B(n,6)$& $S_B(n,7)$\\
		\midrule
		0&1&1\\
		1&2&1&1\\
		2&6&1&4&1\\
		3&24&1&13&9&1\\
		4&116&1&40&58&16&1\\
		5&648&1&121&330&170&25&1\\
		6&4088&1&364&1771&1520&395&36&1\\
		7&28640&1&1093&9219&12411&5075&791&49&1\\
		\bottomrule
\end{tabular}}
\end{table}

A \textit{set partition of type $D$} is actually a signed set partition of $\langle n \rangle$ whose zero-block contains at least two positive elements or no positive elements of $[n]$ (see \cite{br1}). A set partition of type $D$ is called \textit{$D$-type partition}. 
Let $S_D(<n>,k)$ denote the collection of all $D$-type partitions of $\langle n \rangle$ with $2k+1$ blocks. The Stirling numbers of type $D$ of the second kind, denoted by $S_D(n,k)$, is defined to be the size of $S_D(<n>,k)$ (see \cite{br2}).

\begin{ex}
	The following signed partitions
	\begin{enumerate}
		\item [$\bullet$] $\rho=0, \pm{2}, \pm{3}, \pm{5}~|~1/-1~|~4/-4$
		\item [$\bullet$] 
		$\tau=0~|~-1, 2/1,-2~|~3,5/-3,-5~|~4/-4$
	\end{enumerate}
	are $D_5$-type partitions. But, the signed partition $\gamma=0,\pm{5}~|~1, 3, -4/-1,-3, 4~|~2/-2$ is not $D_5$-type partition.
\end{ex}
There exist the following two well-known identities between Stirling numbers of the second kind in types classical, $B$ and $D$ (see Corollary 12 in \cite{br5} and Proposition 3 in \cite{br2'}):
\begin{equation}\label{d1}
	S_B(n,k)=\sum_{i=k}^n 2^{i-k} {n \choose i} S(i,k)
\end{equation}
\begin{equation}\label{d2}
	S_D(n,k)=S_B(n,k)- n 2^{n-1-k} S(n-1,k).
\end{equation}
Thus it is easy to check that the number of signed partitions of $\langle n \rangle$ with only one positive element in its zero-block is 
$$n\sum_{k=0}^{n-1} 2^{n-1-k} S(n-1,k).$$
One can immediately obtain the exponential generating function of $S_D(n,k)$ by using (\ref{d1}) and (\ref{d2}) and Theorem 4.1(a) in \cite{br3} as
$$\sum_{n \geq 0} S_D(n,k)\frac{x^n}{n!}=\frac{1}{2^k k!}(e^x-x)(e^{2x}-1)^k.$$
In OEIS, Bell numbers of type $D$ corresponds to the sequence oeis.org/A039764. Table \ref{tab3} records both Bell numbers of type $D$ and the second kind Stirling numbers in type $D$ for some values of $n$ and $k$.\\

\begin{table}[h]%
	\centering
	\caption{\centering{Bell numbers $D(n)$ and Stirling numbers in type $D$ of the second kind $S_D(n,k)$}\label{tab3}}%
  \scalebox{0.8}{
	\begin{tabular}{c|c||cccccccc}
		\toprule
		n &$D(n)$& $S_D(n,0)$ & $S_D(n,1)$ &$S_D(n,2)$&$S_D(n,3)$& $S_D(n,4)$& $S_D(n,5)$&$S_D(n,6)$& $S_D(n,7)$\\
		\midrule
		0&1&1\\
		1&1&0&1\\
		2&4&1&2&1\\
		3&15&1&7&6&1\\
		4&72&1&24&34&12&1\\
		5&403&1&81&190&110&20&1\\
		6&2546&1&268&1051&920&275&30&1\\
		7&17867&1&869&5747&7371&3255&581&42&1\\
		\bottomrule
\end{tabular}}
\end{table}

\begin{defn}[\cite{br1'}]
	The number of all signed set partitions of the $\langle n\rangle$ is called \textit{Bell numbers of type $B$} and is denoted by $B(n)$. Then we can write 
	$$B(n)=\sum_{k=0}^nS_B(n,k).$$
\end{defn}

We would like to thank Bruce Sagan for calling reference \cite{br1'} to our attention. Upon this, we encountered that Shankar  in \cite{br3'} proved that Bell numbers in type $B$ satisfy tle following recurrence relation
	\begin{equation}\label{33}
		B(n+1)=B(n)+\sum_{k=0}^n 2^k {n \choose k} B(n-k)
	\end{equation}
for all $n \geq 0$. Moreover, Shankar gave the recurrence relation in equation (\ref{33}) in more general form for colored Bell numbers and also established an explicit formula for colored Bell numbers in Corollary 4 of \cite{br3'}. Thus it is clear from \cite{br3'} that for all $n \geq 0$ 
	\begin{equation}\label{13}
B(n)=e^{\frac{-1}{2}}   \left( \sum_{r \geq 0}\frac{(2r+1)^n}{2^r r!} \right ).
	\end{equation}
 In \cite{br1'}, Buck et al. established a bijection between flattened Stirling permutations of order $n+1$ and signed set partitions of $\langle n \rangle$, and then left many open problems and conjectures. Shankar \cite{br3'} has proved all of the conjectures and open problems, except for Problem 3, posed by Buck et al. in \cite{br1'}. Furthermore, we want especially to note that the exponential generating function of $B(n)$ is 
	\begin{equation}\label{03} 
B(x)=e^{\frac{e^{2x}}{2}+x-\frac{1}{2}}
	\end{equation}
which follows easily from Theorem 4.1(a) of Sagan and Swanson, see \cite{br3}.

\begin{prop}\label{rel}
	For all $n \geq 0$, we have
	$$\sum_{k=0}^n (2k+1) S_B(n,k)=\sum_{k=0}^n {n \choose k} 2^k B(n-k).$$
\end{prop}

\begin{proof}
	Taking into account equation (\ref{rec}) and $B(n+1)=\sum_{k=0}^{n+1}S_B(n+1,k)$, then we can write 
	\begin{align*}
		B(n+1)=&\sum_{k=0}^{n+1}S_B(n+1,k)   \\
		=&S_B(n+1,0)+\sum_{k=1}^{n+1} \left( S_B(n,k-1)+(2k+1)S_B(n,k) \right) \\
		=& 1+\sum_{k=0}^{n} S_B(n,k)+\sum_{k=1}^{n+1} (2k+1)S_B(n,k)~~(\textrm{since}~~S_B(n,n+1)=0)\\
		=&B(n)+\sum_{k=0}^{n} (2k+1)S_B(n,k).
	\end{align*}
	Since $B(n+1)=B(n)+\sum_{k=0}^n 2^k {n \choose k} B(n-k)$, we obtain the desired result as
	$$\sum_{k=0}^n (2k+1) S_B(n,k)=\sum_{k=0}^n {n \choose k} 2^k B(n-k).$$
\end{proof}

For example, we illustrate the relation in Proposition \ref{rel} for both $n=2$ and $n=3$ as\\

\begin{enumerate}  
	\item [$\bullet$] $S_B(2,0)+3S_B(2,1)+5S_B(2,2)=18=B(2)+4B(1)+4B(0)$\\
	\item [$\bullet$] $S_B(3,0)+3S_B(3,1)+5S_B(3,2)+7S_B(3,3)=92=B(3)+6B(2)+12B(1)+8B(0)$.
\end{enumerate}

\section{Bell numbers of type $D$}
In this section, we will first introduce Bell numbers of type $D$ and then give their some combinatorial properties such as a recurrence relation, the exponential generating function and an explicit formula.
\begin{defn}
	The number of all $D$-type set partitions of the $\langle n\rangle$ is called \textit{Bell numbers of type $D$} and we will denote it by $D(n)$. Then we can express 
	$$D(n)=\sum_{k=0}^nS_D(n,k).$$
\end{defn}

\begin{thm}\label{99}
	Bell numbers in type $D$ satisfy
	\begin{equation}\label{11}
		D(n+1)=\sum_{i=1}^n  {n \choose i} \sum_{k=0}^{n-i}2^{n-i-k}S(n-i,k)+\sum_{k=0}^n 2^k {n \choose k} D(n-k)
	\end{equation}
	for all $n \geq 0$, where $S(n,k)$ is the classical Stirling number of the second kind.
\end{thm}

\begin{proof}
	Consider $\gamma=P_0~|~P_1/P_2~|~P_3/P_4~|~\cdots~|~P_{2s-1}/P_{2s}$ a signed set partition of $\langle n+1 \rangle$ with $s$ paired non-zero blocks. 
	
	\textbf{Case 1}. Without loss of generality, we may suppose that  $n+1 \in P_1$, and that $|P_1|=k+1$, for some $k,~0 \leq k \leq n$. Then 
	$P_0~|~P_3/P_4~|~ \cdots~|~P_{2s-1}/P_{2s}$ has the form of a signed partition of the remaining $2n-2k$ elements of $\langle n+1 \rangle$, that is, $\langle n+1 \rangle \setminus (P_1 \biguplus P_2)$. There are $D(n-k)$ signed partitions of this set. Thus there are $D(n-k)$ signed partitions of $\langle n+1 \rangle$ in which one paired non-zero block is the $P_1 / P_2$. There are $2^k {n \choose k}$ signed sets of size $k+1$ containing $n+1$. Thus, the total number of $D$-type partitions of $\langle n+1 \rangle$ in which $n+1$ is in a signed set of size $k+1$ is $2^k {n \choose k} D(n-k)$. Summing up over all possible values of $k$, we obtain the second part of the summation in the right hand side of equation (\ref{11}). 
	
	\textbf{Case 2}. 
	Suppose that $n+1 \in P_0$. To create a set partition of type $D$ with zero-block $P_0$ such that $n+1 \in P_0$, we first need to select at least one more element of $[n]$ as a zero-block in a $D$-type partition must contain at least two positive elements or no positive elements of $[n]$. So we can do this in ${n\choose i}$ possible ways, where $i$ ranges from $1$ to $n$. In other words, there are ${n\choose i}$ possible ways to do this. This construction provides all possible $P_0$ such that $n+1 \in P_0$. Now we want to calculate the number of all set partitions of type $D$ such that $n+1 \in P_0$. The number of distribution of the remaining $n-i$ elements into $k$ non-empty paired blocks is $S(n-i,k)$. Since the smallest element in each of $k$ paired blocks will be positive, and the remaining $n-k-i$ elements in those $k$ blocks can be positive or negative, then there are $2^{n-i-k}$ possible ways to pick the signs of these remaining values. Since $k$ will vary from $0$ to $n-i$, we then conclude that the number of $D$-type set partitions of $\langle n+1 \rangle$ for this case is exactly $\sum_{i=1}^n {n \choose i} \sum_{k=0}^{n-i}2^{n-i-k}S(n-i,k)$. Therefore, we conclude the first part of the summation in the right hand side of equation (\ref{11}).
\end{proof}

More clearly, we can give the following example.

\begin{ex}
We can see how the formula in (\ref{11}) works for small values of $n$.
\begin{enumerate}
\item [$\bullet$] $D(3)={2 \choose 1}\Bigl\{2S(1,0)+S(1,1) \Bigl\}+{2 \choose 2}S(0,0)+{2 \choose 0}D(2)+2{2 \choose 1}D(1)+2^2 {2 \choose 2} D(0)=15$
\item [$\bullet$] $D(4)={3 \choose 1}\Bigl \{2^2S(2,0)+2S(2,1)+S(2,2) \Bigl\}+{3 \choose 2}\Bigl\{2S(1,0)+S(1,1)\Bigl\}+{3 \choose 3}S(0,0)+{3 \choose 0}D(3)+2{3 \choose 1}D(2)+2^2 {3 \choose 2} D(1)+2^3 {3 \choose 3} D(0)=72$
\item [$\bullet$] $D(5)={4 \choose 1}\Bigl\{2^3S(3,0)+2^2 S(3,1)+2S(3,2)+S(3,3)\Bigl\}+{4 \choose 2}\Bigl\{2^2S(2,0)+2S(2,1)+S(2,2)\Bigl\}+{4 \choose 3}\Bigl\{2 S(1,0)+S(1,1)\Bigl\}+{4 \choose 4}\Bigl\{S(0,0)\Bigl\}+{4 \choose 0}D(4)+2{4 \choose 1}D(3)+2^2 {4 \choose 2} D(2)+2^3 {4 \choose 3} D(1)+2^4 {4 \choose 4} D(0)=403$.
		
	\end{enumerate}
\end{ex}

Now we will derive the exponential generating functions of Bell numbers in type $D$ in the following theorem:

\begin{thm}\label{00}
	The exponential generating function of $D(n)$ is
	$$D(x)=e^{\frac{e^{2x}-1}{2}}(e^x-x).$$
\end{thm}

\begin{proof}
	Let the exponential generating function of $D(n)$ be 
	$$D(x)=\sum_{n \geq 0} D(n)\frac{x^n}{n!}.$$ 
	If we take the derivative of both sides of the above expression with respect to $x$, then we obtain 
	$$D'(x)=\sum_{n \geq 1} D(n)\frac{x^{n-1}}{(n-1)!}=\sum_{n \geq 0} D(n+1)\frac{x^{n}}{n!}.$$
	We obtain from Theorem 4.7 in \cite{br1'} that $\sum_{i=1}^n  {n \choose i} \sum_{k=0}^{n-i}2^{n-i-k}S(n-i,k)=B(n)-\sum_{k=0}^{n}2^{n-k}S(n,k)$. Thus we conclude from the fact $\sum_{n \geq 0} S(n,k) t^k\frac{x^{n}}{n!}=e^{t(e^x-1)}$ in Section 4.1 in \cite{br3} that 
	\begin{align*}
		D'(x)=&\sum_{n \geq 0} D(n+1)\frac{x^{n}}{n!} \\
		=&\sum_{n \geq 0} \left( B(n)-\sum_{k=0}^{n}2^{n-k}S(n,k)\right) \frac{x^n}{n!}+\sum_{n \geq 0} \left( \sum_{k=0}^n 2^k {n \choose k} D(n-k)\right) \frac{x^{n}}{n!}\\
		=&B(x)-\left( \sum_{n, k \geq 0}2^{n-k}S(n,k) \frac{x^n}{n!}\right)+\left(\sum_{n \geq 0} 2^n \frac{x^n}{n!}\right)\left(\sum_{n \geq 0} D(n)\frac{x^n}{n!}\right)\\
		=&B(x)-e^{\frac{e^{2x}-1}{2}}+e^{2x}D(x)\\
		=&e^{\frac{e^{2x}-1}{2}+x}+e^{\frac{e^{2x}-1}{2}}+e^{2x}D(x).
	\end{align*}
	
	When solving the initial value problem $D'(x)=e^{\frac{e^{2x}-1}{2}+x}+e^{\frac{e^{2x}-1}{2}}+e^{2x}D(x),~~D(0)=1$, this yields the exponential generating function of $D(n)$ as $D(x)=e^{\frac{e^{2x}-1}{2}}(e^x-x).$
\end{proof}

As a result of Theorem \ref{00}, we will give an explicit formula of Bell numbers of type $D$ in the following corollary.
\begin{cor}
	For $n \geq 0$, we have
	\begin{equation*}
		D(n)=e^{\frac{-1}{2}} \sum_{r \geq 0}\frac{1}{2^r r!}[(2r+1)^n-n(2r)^{n-1}].
	\end{equation*}
\end{cor}

\begin{proof}
	Let us consider that $D(x)=e^{\frac{e^{2x}-1}{2}}(e^x-x)$ from Theorem \ref{00}. Thus, we obtain from equations (\ref{13}) and (\ref{03}) that
	\begin{align*}
		D(x)&=\sum_{n \geq 0} D(n)\frac{x^n}{n!}=e^{\frac{e^{2x}-1}{2}}(e^x-x) \\
		&=B(x)-e^{\frac{-1}{2}}xe^{\frac{e^{2x}}{2}}\\
		&=e^{\frac{-1}{2}}  \sum_{n \geq 0} \left( \sum_{r \geq 0}\frac{(2r+1)^n}{2^r r!} \right ) \frac{x^n}{n!}-e^{\frac{-1}{2}}   \sum_{n \geq 0} \left(  \sum_{r \geq 0}\frac{(2r)^n}{2^r r!} \right ) \frac{x^{n+1}}{n!}\\
		&=e^{\frac{-1}{2}}  \sum_{n \geq 0} \left( \sum_{r \geq 0}\frac{(2r+1)^n}{2^r r!} \right ) \frac{x^n}{n!}-e^{\frac{-1}{2}}   \sum_{n \geq 0} \left(  \sum_{r \geq 0}\frac{n(2r)^{n-1}}{2^r r!} \right ) \frac{x^{n}}{n!}\\
		&=e^{\frac{-1}{2}} \sum_{r \geq 0}\frac{1}{2^r r!}[(2r+1)^n-n(2r)^{n-1}] \frac{x^{n}}{n!}.
	\end{align*}
	Comparing the coefficients of $\frac{x^{n}}{n!}$ for each $n\geq 0$, we then obtain the desired formula.
\end{proof}

\section{Questions:}
	\textit{\textbf{1.} It is well-known from \cite{du2022} that there is the relationship $$A(r)=\frac{1}{n!}\sum_{k=0}^n \binom{n}{k}k^rd_{n-k}$$ between derangement numbers $d_n$ and Bell numbers $A(n)$ in $S_n$. Mez\"o and Ram\'irez in \cite{mezo2019} established a nice connection between colored derangements and $r$-Bell numbers. Thus, what is the relationship between Bell numbers of type $D$ and the numbers of derangements of type $D$ studied in both \cite{chow2023} and \cite{Ji2025}?} \\
	\textit{\textbf{2.} In \cite{br3''} Sagan and Swanson constructed a set partition of type $G(m,p,n)$. Thus, it is a natural question to ask here when considering the set partition of type $G(m,p,n)$, then how exactly can Bell numbers of type $G(m,p,n)$ be defined?} \\
	\textit{\textbf{3.} What kind of relationship exists between Bell numbers of type $D$ and flattened Stirling permutations of order $n+1$ in sense of \cite{br1'}?}

\end{document}